\documentclass[bj,authoryear]{imsart}

\usepackage[ruled,noend]{algorithm2e}
\usepackage{bbm}
\startlocaldefs
\theoremstyle{plain}
\newtheorem{thm}{Theorem}[section]
\newtheorem{lem}[thm]{Lemma}
\newtheorem{prop}[thm]{Proposition}

\theoremstyle{remark}
\newtheorem{rem}[thm]{Remark}

\newtheorem{defn}[thm]{Definition}

\newtheorem{ex}[thm]{Example}

\DeclareMathOperator{\var}{var}

\newcommand{\E}{\mathbb{E}}
\newcommand{\V}{\mathbb{V}}
\newcommand{\cE}{\mathcal{E}}

\newcommand{\cH}{\mathcal{H}}

\renewcommand{\P}{\mathbb{P}}

\newcommand{\R}{\mathbb{R}}

\newcommand{\cY}{\mathcal{Y}}

\def\proofname{{\it Proof.}}

\endlocaldefs

\begin{document}
\begin{frontmatter}
\title{Learning latent tree models with small query complexity}
%\title{A sample article title with some additional note\thanksref{t1}}
\runtitle{Latent tree models with small query complexity}
%\thankstext{T1}{A sample additional note to the title.}

\begin{aug}
\author[A]{\inits{F.}\fnms{Luc}~\snm{Devroye}\ead[label=e1]{lucdevroye@gmail.com}}
\author[B,C]{\inits{S.}\fnms{G\'{a}bor}~\snm{Lugosi}\ead[label=e2]{gabor.lugosi@upf.edu}}
\author[B]{\inits{T.}\fnms{Piotr}~\snm{Zwiernik}\ead[label=e3]{piotr.zwiernik@utoronto.ca}}
%%%%%%%%%%%%%%%%%%%%%%%%%%%%%%%%%%%%%%%%%%%%%%
%% Addresses                                %%
%%%%%%%%%%%%%%%%%%%%%%%%%%%%%%%%%%%%%%%%%%%%%%
\address[A]{School of Computer Science, McGill University, Montreal, Canada\printead{e1}}

\address[B]{Department of Economics and Business,
Pompeu  Fabra University, Barcelona, Spain;\\
Barcelona School of Economics\printead{e2,e3}}

\address[C]{Department,
ICREA, Pg. Llu\'{i}s Companys 23, 08010 Barcelona, Spain\printead{}}

%\address[D]{Department of Statistical Sciences, University of Toronto, Canada\printead{e3}}

\end{aug}

\begin{abstract}
We consider the problem of structure recovery in a graphical model of a tree where some variables are latent. Specifically, we focus on the Gaussian case, which can be reformulated as a well-studied problem: recovering a semi-labeled tree from a distance metric. We { introduce randomized} procedures that achieve query complexity of optimal order. Additionally, we provide statistical analysis for scenarios where the tree distances are noisy. The Gaussian setting can be extended to other situations, including the binary case and non-paranormal distributions.\end{abstract}


\end{frontmatter}

\section{Introduction}

First discussed by Judea Pearl as tree-decomposable distributions to generalize star-decomposable distributions such as the latent class model \cite[Section 8.3]{pearl}, latent tree models are probabilistic graphical models defined on trees, where only a subset of variables is observed. 
%Nevin L. Zhang and co-authors 
\cite{zhang2004hierarchical,zhang2017latent,mourad2013survey} extended our theoretical understanding of these models. We refer to \cite{zwiernik2018latent} for more details and references.

Latent tree models have found wide applications in various fields. They are used in phylogenetic analysis, network tomography, computer vision, causal modelling, and data clustering. They also contain other well-known classes of models like hidden Markov models, the Brownian motion tree model, the Ising model on a tree, and many popular models used in phylogenetics. In generic high-dimensional problems, latent tree models can be useful in various ways. They share many computational advantages of observed tree models but are more expressible. Latent tree models have been used for hierarchical topic detection \citep{come2021hierarchical} and clustering.

In phylogenetics, latent tree models have been used to reconstruct the tree of life from the genetic material of surviving species. They have also been used in bioinformatics and computer vision. Machine-learning methods for models with latent variables attract substantial attention from the research community. Some other applications include latent tree models and novel algorithms for high-dimensional data (\cite{chen2019novel}), and the design of low-rank tensor completion methods (\cite{zhang2022low}).

\subsection*{Structure learning}

The problem of structure or parameter learning for latent tree models has been extensively studied. A seminal work in this field is by \cite{choi2011learning}, where two consistent and computationally efficient algorithms for learning latent trees were proposed. The main idea is to use the link between a broad class of latent tree models and tree metrics studied extensively in phylogenetics, a link first established by \cite{pearl} for binary and Gaussian distributions.  This has been extended to symmetric discrete distributions by \cite{choi2011learning}. The proof in \cite[Section 2.3]{zwiernik2018latent} makes it clear that the only essential assumption is that the conditional expectation of every node given its neighbour is a linear function of the neighbour. 

From the statistical perspective, testing the corresponding algebraic restrictions on the correlation matrix was studied by \cite{shiers2016correlation}.  \cite{sturma2022testing} revisited this problem.  For the machine learning perspective, see \cite{jaffe2021spectral,zhou2020learning,anandkumar2011spectral,huang2020guaranteed,aizenbud2021spectral,kandiros2023learning}. Most of these papers build on the idea of local recursive grouping as proposed by \cite{choi2011learning}.  In particular, they all start by computing all distances between all observed nodes in the tree. 

\subsection*{Our contributions} Our work is motivated by novel applications where the dimension $n$ is so large that computing all the distances---or all correlations between observed variables---is impossible. We propose a randomized algorithm that queries the distance oracle and show that the expected query time for our algorithm is $O(\Delta n \log_\Delta(n))$, where $n$ is the number of observed variables and $\Delta$ is the maximal degree of the underlying tree. A special case of our problem is the problem of phylogenetic tree recovery. In this case, our approach resembles other phylogenetic tree recovery methods that try to minimize query complexity (see  \cite{afshar2020reconstructing} for references) and seems to be the first such method that is asymptotically optimal; see  \cite{zhang2003complexity} for the matching lower bound. More importantly, our algorithm can deal with general latent tree models and in this context, it is again the first such algorithm.  As we show in the last section, our algorithm can be easily adjusted to the case of noisy oracles, which is relevant in statistical practice.

\section{Preliminaries and basic results}

\subsection{Trees and semi-labeled trees}
A \textit{tree} $T=(V, E)$ is a connected undirected graph with no cycles. In particular, for any two $u,v\in V$ there is a unique path between them, which we denote by $\overline{uv}$. A vertex of $T$ with only one neighbour is called a \textit{leaf}. A vertex of $T$ that is not a leaf is called an \textit{inner vertex} or \textit{internal node}. An edge of $T$ is \textit{inner} if both ends are inner vertices; otherwise, it is called \textit{terminal}.  A connected subgraph of $T$ is called a \textit{subtree} of $T$. A rooted tree $(T,\rho)$ is simply a tree $T=(V,E)$ with one distinguished vertex $\rho\in V$.

A tree $T$ is called a semi-labeled tree with labeled nodes $W\subseteq V$ if every vertex of $T$ of degree $\leq 2$ lies in $W$\footnote{This differs slightly from the regular definition of semi-labeled trees (or X-trees) in phylogenetics, where regular nodes can get multiple labels; see \cite{semple2003phylogenetics}.}. We say that $T$ is a phylogenetic tree if $T$ has no degree-$2$ nodes and $W$ is exactly the set of leaves of $T$.   If $v\in W$, then we say that $v$ is \textit{regular} and we depict it by a solid vertex. If $|W|=n$ then we typically label the vertices in $W$ with $[n]=\{1,\ldots,n\}$. A vertex that is not labeled is called \textit{latent}. An example semi-labeled tree is shown in Figure \ref{fig:SemiLabTree1}.

\begin{figure}[htp!]
\centering
\includegraphics{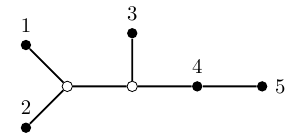}
\caption{A semi-labeled tree with five regular nodes labeled with $\{1,2,3,4,5\}$ and two latent nodes.}\label{fig:SemiLabTree1}
\end{figure}

By definition, all latent nodes are internal and have degree $\geq 3$. Therefore, if $|W|=n$, then $|T| \leq 2n+1$.

\subsection{Tree metrics}

Consider metrics on the discrete set $W$ induced by semi-labeled trees in the following sense. Let $T=(V, E)$ be a tree and suppose that $\ell:E\rightarrow \R_{+}$ is a map that assigns positive lengths to the edges of $T$. For any pair $u,v\in V$ by $\overline{uv}$ we denote the path in $T$ joining $u$ and $v$. We now define the map $d_{T,\ell}:V\times V\rightarrow \R$ by setting, for all $u,v\in V$,
$$
d_{T,\ell}(u,v)=\left\{\begin{array}{ll}
\sum_{e\in \overline{uv}} \ell(e),&\mbox{ if } u\neq v,\\
0, & \mbox{otherwise.}
\end{array}\right.
$$ 
Suppose now we are interested only in the distances between the regular vertices. 
\begin{defn}\label{def:treemetric}
A function $d: \,[n]\times [n]\rightarrow\R$ is called a \textit{tree metric} if there exists a semi-labeled tree $T=(V,E)$ with $n$ regular nodes $W$ and a (strictly) positive length assignment $\ell: E\rightarrow \R_{+}$ such that for all $i,j\in W$
$$
d(i,j)\;=\;d_{T,\ell}(i,j).$$ 
\end{defn}
\begin{ex}\label{ex:quartetMetric}Consider a quartet tree\index{tree!quartet} with edge lengths as indicated on the left in Figure~\ref{fig:quartetMetric}. The distance between vertices $1$ and $3$ is $d({1,3})=2+5+2.5=9.5$ and the whole distance matrix is given on the right in Figure~\ref{fig:quartetMetric}, where the dots indicate that this matrix is symmetric. 
\end{ex}
\begin{figure}[htp!]
\begin{minipage}{8cm}
	\includegraphics{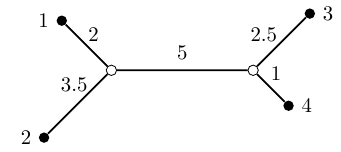}\end{minipage}
  \begin{minipage}{4cm}
  	$\begin{bmatrix}
0 & 5.5 & 9.5 & 8\\
\cdot & 0 & 11 & 9.5\\
\cdot & \cdot & 0 & 3.5\\
\cdot & \cdot & \cdot & 0
\end{bmatrix},$
  \end{minipage}
\caption{A metric on a quartet tree.}\label{fig:quartetMetric}
\end{figure}

It is easy to describe the set of all possible tree metrics.
\begin{defn}\label{def:4pointcond}
We say that a map ${d}:\, [n]\times [n]\to \R$ satisfies the \textit{four-point condition}\index{four-point condition} if for every four (not necessarily distinct) elements $i,j,k,l\in [n]$,
$$
{d}({i,j})+{d}({k,l})\quad\leq\quad \max\left\{\begin{array}{l}
{d}({i,k})+{d}({j,l})\\
{d}({i,l})+{d}({j,k}).
\end{array}\right.
$$
\end{defn}
Since the elements $i,j,k,l\in [n]$ in Definition \ref{def:4pointcond} need not be distinct, every such map is a metric on $[n]$ given that ${d}({i,i})=0$ and ${d}({i,j})={d}({j,i})$ for all $i,j\in [n]$. The following fundamental theorem links tree metrics with the four-point condition.

\begin{thm}[Tree-metric theorem, \cite{buneman1974note}]\label{th:3metric}
Suppose that \mbox{${d}:\, [n]\times [n]\to \R$} is such that ${d}({i,i})=0$ and ${d}({i,j})=d({j,i})$ for all $i,j\in [n]$. Then, ${d}$ is a tree metric on $[n]$ if and only if it satisfies the four-point condition. Moreover, a tree metric uniquely determines the defining semi-labeled tree and edge lengths.\end{thm}

\noindent Note that the assumption about strictly positive lengths of each edge in Definition~\ref{def:treemetric} is crucial for uniqueness in Theorem~\ref{th:3metric}.

\subsection{Recovering a tree from a tree metric}

Our problem is to recover the semi-labeled tree $T$  based on a few queries of the distance matrix $D=[d(u,v)]$, which contains the distances between the regular nodes of $T$ (\cite{Abdi1990, Warnow1999, WCT1999, Fel2004}). 
We consider the \textit{query complexity}, which measures the number
of queries of the distance matrix, also called queries of the distance oracle.  Trivially, we can
reconstruct the tree with query complexity ${n \choose 2}$
(\cite{BPS1990, Boesch1968, CR1989, Gusfield1997, HY1964, waterman1977additive, NWRE2005}).

To provide a more refined analysis of the query complexity, we introduce the maximum degree of the tree:
$$\Delta\;:=\;\max_{i\in V}{\rm degree}(i).$$ When the query complexity is jointly measured in terms
of $n$ and $\Delta$, a lower bound for both worst-case and
expected query complexity is $\Omega (\Delta n \log_\Delta (n))$ 
(\cite{zhang2003complexity}).  For another proof of the lower bound for the expected complexity, see \cite{bastide2024}.

When $T$ is a phylogenetic tree ($W$ is the set of leaves of $T$), the distance queries are sometimes referred to as ``additive queries'' (\cite{waterman1977additive}). This case has been extensively studied in the literature; see \cite{Jan2016} for a recent overview. When the maximum degree $\Delta$ is bounded, \cite{hein1989optimal} demonstrated that the problem can be solved using $O(n \log(n))$ distance queries. When $\Delta$ is unbounded, \cite{CR1989} proposed an algorithm that was claimed to achieve a query complexity matching the lower bound for trees where all edge weights are equal to $1$. However, as noted by \cite{reyzin2007longest}, the algorithm's actual runtime is $O(n^{3/2} \sqrt{\Delta})$.

Complementing this work on phylogenetic recovery, \cite{kannan1996determining} proposed an algorithm with $O(\Delta n \log(n))$ query complexity for the noisy-ultrametric model, a special computational framework that is not directly related to ``additive queries''. In the same computational model, \cite{brodal2001complexity} presented an algorithm matching the lower bound of $O(\Delta n \log_\Delta(n))$; see also \cite{kao1999balanced} for related results. For additional references on this and other specialized computational models, we refer to \cite{afshar2020reconstructing}. Further discussions and generalizations of phylogenetic tree reconstruction can be found in \cite{GMS2012, DMR2009, DMS2006, Bun1971, Csuros2002, GJLO1999}.

Our objective is to develop a randomized algorithm with expected query complexity $O(\Delta n \log_\Delta (n))$ for the general semi-labeled tree reconstruction problem, regardless of how $\Delta$ varies with $n$. Similar to the seminal work of \cite{hein1989optimal}, we focus on pairwise distance queries. This problem generalizes the phylogenetic tree reconstruction problem, as it allows internal non-latent nodes to have degree two, and regular nodes are not necessarily leaves. Consequently, methods designed specifically for phylogenetic tree reconstruction are no longer directly applicable. 

\section{The new algorithm}

The proposed algorithm uses a randomized version of divide-and-conquer.
We will use the notion of a \textit{bag} $B \subseteq V$. 
{ The algorithm maintains a queue, consisting of sets of bags. Initially, there is only one bag, containing all regular nodes, that is, $B=W$.
The procedure takes a bag $B$.
}
One node in this set, denoted by $\rho(B)$, is marked as a
representative, which can be thought of as a root. 
A set of edges that jointly form a tree is called a \textit{skeleton} and is typically denoted by the mnemonic $S$.  Our algorithm starts with an empty skeleton and incrementally constructs the skeleton of the sought tree,
which we call the tree \textit{induced} by $B$.

\begin{algorithm}[H]\label{alg:tree}
\caption{Outline of our algorithm}
Let $\kappa = \Delta $ \; 
Pick a node $u \in W$, and set $\rho(W) \leftarrow u$ (note: $W$ is now a bag)\;
Make an empty queue $Q$ \;
Add $W$ to $Q$ \;
Set $S \leftarrow \emptyset$ \;
\While{$|Q|>0$}
{
Remove bag $B$ from the front of $Q$ \;
\If{$|B| \le \kappa$}{
Query all ${|B| \choose 2}$ distances between nodes in $B$\;
Find $S^*$, the full skeleton for the tree induced by $B$ \;
$ S \leftarrow S \cup S^*$\; 
}
\Else{
Apply procedure \textsc{bigsplit} $(B)$.
This procedure outputs a skeleton $S^*$ (connecting nodes from $B$ and possibly latent nodes)
and bags $B_1,\ldots,B_k$ where $B_i$ overlaps with the nodes of the skeleton in $\rho(B_i)$ only,
and are non-overlapping otherwise (i.e., $(B_i \setminus \rho(B_i)) \cap (B_j \setminus \rho(B_j))  = \emptyset$) \;
$ S \leftarrow S \cup S^*$ \;
Add $B_1,\ldots,B_k$ to the rear of $Q$ \; }}
Return the skeleton $S$ \;
\end{algorithm}

The procedure \textsc{bigsplit} takes a bag $B$ and a random set
of nodes in it, $u_1,\ldots, u_\kappa $, and forms the subtree that connects $u_1,\ldots, u_\kappa $ and $\rho(B)$. The edges of this subtree give the skeleton that is the output. The remaining nodes of $B$ are collected in bags that ``hang'' from the skeleton. The representatives of these bags are precisely those nodes where the bags overlap the skeleton.  Note that the skeleton may contain latent nodes not originally in $B$. Within \textsc{bigsplit}, all representatives of the hanging bags have their distances to all nodes in their bags queried, so for all practical purposes, the newly discovered latent nodes act as regular nodes.
The bags become smaller as the algorithm proceeds, which leads
to a logarithmic number of rounds. 
{The main result of the paper is the following 
theorem, whose proof is given in Section~\ref{sec:analysis} below.}

\begin{thm}\label{th:main}
Given a distance oracle $D$ between the regular vertices of a semi-labeled tree $T$, Algorithm~\ref{alg:tree} with parameter $\kappa = \Delta$ correctly recovers the induced tree with expected query complexity $O(\Delta n\log_\Delta (n))$.
\end{thm}

{The procedure \textsc{bigsplit} uses two sub-operations, called \textsc{basic} and \textsc{explode} that we describe next.
}

\subsection{The ``basic'' operation}\label{sec:basic}

Let $B$ be a bag with representative $\rho = \rho(B)$,
and let $\alpha \in B$ be a distinct regular vertex.
In our basic step, we query $d(v,\alpha)$ for all $v\in B$, and set
\begin{equation}\label{eq:Dv}
D(v) = d(v,\alpha) - d(v,\rho).
\end{equation}
We group all nodes $v$ according to the different values $D(v)$
that are observed.
Ordering the sets in this partition of $B$ from small to large
value of $D(\cdot)$, we obtain bags $B_1,\ldots, B_k$.  It is
clear that $\rho \in B_k$ and $\alpha \in B_1$. 
Within each $B_i$, we let $u_i$ be the node of $B_i$
closest to $\alpha$. If 
$$
d(u_i, \alpha) + d(u_i, \rho) = d(\alpha, \rho),
$$
then $u_i$ (a regular node) is on the path from $\alpha$ to $\rho$
in the induced tree.  We set $\rho(B_i) = u_i$.
If, however, 
$$
d(u_i, \alpha) + d(u_i, \rho) > d(\alpha, \rho),
$$
then we know that there must be a latent node $w_i$
that connects the $(\alpha, \rho)$ path to the nodes in $B_i$.
In fact, for all $v \in B_i$, we have
\begin{equation}\label{eq:newreg}
d(v, w_i) = \frac{1}{2} \left( d(v,\alpha) + d(v, \rho) - d(\alpha,\rho) \right).
\end{equation}
These values can be stored for further use.
So, we add $w_i$ to $B_i$ and define $\rho(B_i) = w_i$.
In this manner, we have identified $S^*$, the part of the final skeleton
that connects $\alpha$ with $\rho$:
$$
(\rho (B_1), \rho(B_2)),
\ldots,
(\rho (B_{k-1}), \rho(B_k)).
$$

\begin{figure}[hbt!]
\centering
\includegraphics[scale=.8]{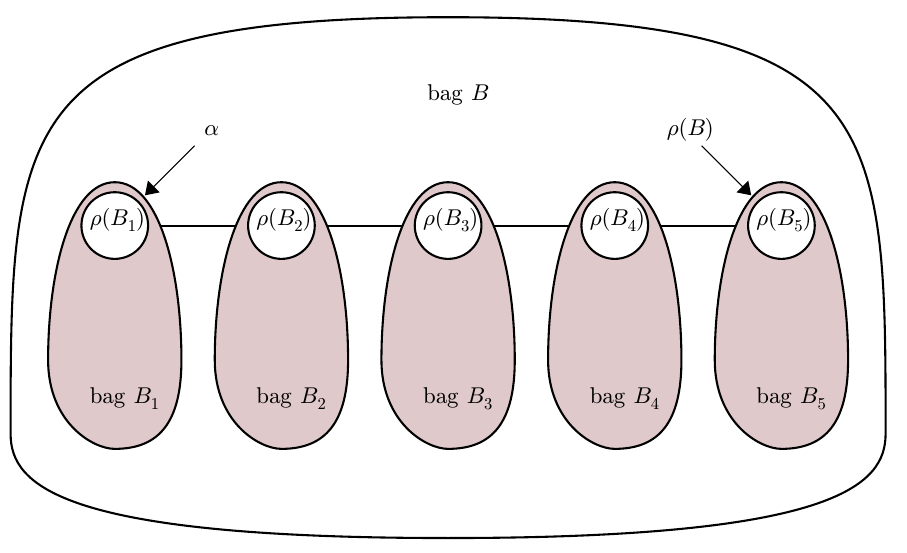}
\caption{In a basic operation, a bag $B$ with root $\rho(B)$
and query node $\alpha$ is decomposed into several smaller bags,
with root nodes on the skeleton.}
\label{fig:basic}
\end{figure}

\begin{algorithm}[H]\label{alg:basic}
Input: a bag $B$, and $\alpha\in B$\;
Set $\rho=\rho(B)$\;
\For{$v\in B$}{
Compute $D(v)$ in \eqref{eq:Dv}\; 
}
Assign all $v \in B$ to bags $B_1,\ldots, B_k$ according to decreasing values of $D(v)$\; 
\For{$i=1,\ldots,k$}{
$u_i=\arg\min_{v\in B_i} d(v,\alpha)$\;
\If{$d(\rho,u_i)+d(u_i,\alpha)=d(\rho,\alpha)$}{$\rho(B_i)=u_i$}
\Else{Identify latent node $w_i$\;
Add $w_i$ to $B_i$\;
Set $\rho(B_i)=w_i$\;
Update the distance oracle by calculating  $d(w_i,v)$ for all $v\in B_i$ using \eqref{eq:newreg}\;}
}
Return $B_1,\ldots,B_k$ and the skeleton $(\rho(B_1), \rho(B_2)), \ldots, (\rho(B_{k-1}), \rho(B_k))$\;
%\vspace*{.2cm}
\caption{\textsc{basic}$(B,\alpha)$}
\end{algorithm}

The query complexity of Algorithm
\ref{alg:basic} is bounded by $2|B|-3$.

\subsection{The operation ``explode''}\label{sec:explode}

Another fundamental operation, called \textit{explode}, decomposes
a bag $B$ into smaller bags $B_1,\ldots,B_k$---all having the same representative $\rho (B)$---according to the different subtrees of $\rho(B)$ that are part of the tree induced by $B$.
For arbitrary nodes $u \not= v \in B$, $u,v \not= \rho(B)$,
we note that $u$ and $v$ are in the same subtree
if and only if
$$
d(u,v) < d(u, \rho(B)) + d(\rho(B),v).
$$
Thus, in query time $O(|B|)$, we can determine the nodes that are in the same subtree
as $u$. 
Therefore, we can partition all nodes of $B \setminus \{\rho(B)\}$ into disjoint subtrees of $\rho(B)$ (without constructing these trees yet)
in query time at most $|B| \Delta$ by peeling off
each set in the partition in turn.  These sets are output as bags
denoted by $B_1,\ldots, B_k$.

\begin{figure}[hbt!]
\centering
\includegraphics[scale=0.9]{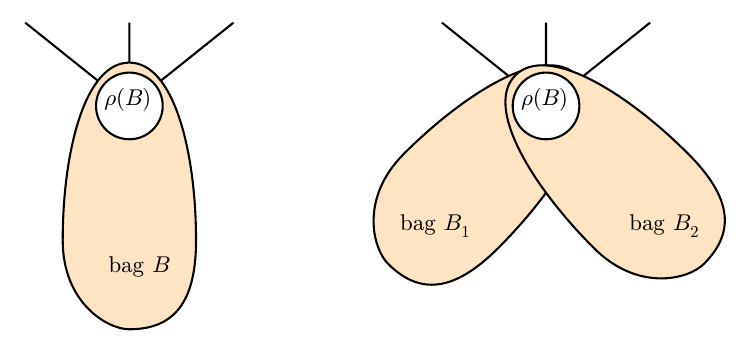}
\caption{In the explode operation, a bag $B$ with root $\rho(B)$
is decomposed into several smaller bags,
each with root $\rho(B)$. All other nodes of $B$ end up in only one
of the smaller bags.}
\label{fig:explode}
\end{figure}

\begin{algorithm}[H]\label{alg:explode}
Input: a bag $B$ \;
\While{$|B|>1$}
{Take $u \in B$, $u \not= \rho(B)$\;
Set $B^* = \{ v \in B, v\not=\rho(B): d(u,v) < d(u,\rho(B))+d(v,\rho(B)) \} \cup \{ \rho(B) \}$\;
Set $\rho(B^*) = \rho(B)$\;
Output $B^*$ \;
Set $B \leftarrow \{ \rho (B) \} \cup ( B \setminus B^* )$\;
}
\caption{\textsc{explode}$(B)$}
\end{algorithm}

The query complexity of Algorithm \ref{alg:explode} is bounded by $(|B| -1)\Delta$.

\subsection{The procedure bigsplit}\label{sec:round}

We are finally ready to provide the details of the procedure
\textsc{bigsplit}, which takes as input a bag $B$
and distinct nodes $u_1,\ldots,u_k$ 
{ in $B$} not equal to $\rho(B)$.

\begin{algorithm}[H]\label{alg:bigsplit}
\caption{\textsc{bigsplit}$(B)$}
Let $\mathcal{C} = \{ B \}$  \;
Let $S$ be a skeleton. Initially, $S = \emptyset$ \;
Let $M = \max_{B' \in \mathcal{C}} |B'|$ \;
\While{$M > |B|/\sqrt{\Delta}$}
{Sample uniformly at random and without replacement nodes $u_1,\ldots, u_\kappa $ from $B \setminus \{ \rho(B) \}$\; 
Let $\mathcal{B}$ be a collection of bags. Initially, $\mathcal{B} = \{ B \}$\;
Let $S$ be a skeleton. Initially, $S = \emptyset$ \;
\For{$i=1$ to $\kappa$}
{
Find the bag $B^*$ in $\mathcal B$ to which $u_i$ belongs \;
\If{$u_i \not= \rho(B^*)$}
{
Apply \textsc{basic}$(B^*, u_i)$, which outputs a skeleton $S^*$ and bags $B_1,\ldots,B_k$ \;
${\mathcal{B}} \leftarrow {\mathcal{B}} \setminus \{ B^* \}$\;
${\mathcal{B}} \leftarrow {\mathcal{B}} \cup \cup_{j=1}^k \{ B_j \}$\;
$S \leftarrow S \cup S^*$ \;
}
}
Let $\mathcal{C}$ be a collection of bags. Initially, $\mathcal{C}$ is empty\;
\For{all $B \in {\mathcal{B}}$}
 {
 \textsc{explode}$(B)$, which leaves output $B_1,\ldots,B_\ell$ \;
 Add $B_1,\ldots,B_\ell$ to $\mathcal{C}$\;
 }
 Let $M = \max_{B' \in \mathcal{C}} |B'|$ \;
 }
Output $S$ \;
Output all bags in $\mathcal{C}$ \;
\end{algorithm}

\begin{figure}[hbt!]
\centering
\includegraphics{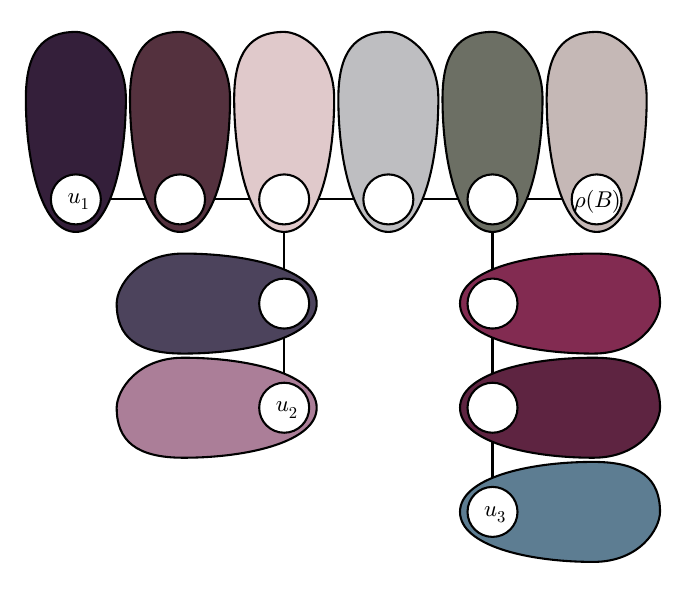}
\caption{Illustration of \textsc{bigsplit}$(B)$. Each bag anchored on the skeleton is further exploded into smaller bags (not shown).}
\label{fig:bigsplit}
\end{figure}

\section{The query complexity}\label{sec:analysis}

\subsection{Proof of Theorem~\ref{th:main}}

For a bag $B$, and nodes $u_1,\ldots,u_\kappa$ 
taken uniformly at random without replacement
{ from $B\setminus \{\rho(B)\}$},
let $M$ be the size of the largest bag output by 
\textsc{bigsplit} $(B)$,
where $|B| \ge \kappa+1$, and $\kappa = \Delta$.
Lemma~\ref{lem:main} below shows that
$$
\P \left\{ M > \frac{|B|}{\sqrt{\Delta}} \right\} \le \frac{1}{2}.
$$
The while loop in \textsc{bigsplit} is repeated until 
$M \le {|B|}/{\sqrt{\Delta}}$.
Thus, we can visualize the algorithm in rounds, starting with $W$.
In other words, in the $r$-th round, we apply \textsc{bigsplit}
to all bags that have been through a \textsc{bigsplit} $r-1$ times.
In every round, all bags of the previous round are reduced in size by
a factor of $1/\sqrt{\Delta}$. Therefore, there are 
$\le \log_{\sqrt{\Delta}} (n) = 2 \log_\Delta (n)$ rounds.
The query complexity of one \textsc{bigsplit} without 
the \textsc{explode} operations is at most
$(\kappa +1)(2|B|-1)$. The total query complexity due to all \textsc{explode} operations
is at most $\Delta (|B|-1)$, for a grand total bounded by
$$
\kappa+1+ (|B|-1) \times (2+ 2\kappa + \Delta).
$$
Then the (random) query complexity 
for splitting bag $B$ is bounded by
\begin{equation}\label{eq:bagsplit}
X(\kappa +1)+ X (|B|-1)  \times (2+ 2\kappa + \Delta)
= X(\Delta+1) + X (|B|-1) \times (2+3 \Delta)~,
\end{equation}
where $X$ is geometric $(1/2)$.
The expected value of this is $2(\Delta+1) + 2 (|B|-1)  (2+ 3 \Delta)$.
Summing over all bags $B$ that participate in one round
yields an expected value bound of $O(\Delta)$ times
the sum of $|B|-1$ over all participating $B$.  But the bags
do not overlap, except possibly for their representatives. Hence
the sum of all values $(|B|-1)$ is at most $n$, as each proper
item in a bag is a regular node. As $\kappa = \Delta$,
the expected cost of one round of splitting is at most
$2(\Delta+1) + n  (4+ 6 \Delta)$.

There is another component of the query complexity due to the part in which we construct the induced tree for a bag $B$ when $|B| \le \kappa$.  A bag $B$ dealt with in this manner is called final.  So the total query complexity becomes the sum of ${|B| \choose 2}$ computed over all final bags of size at least two. Let the sizes of the final bags be denoted by $n_i$.  Noting that $\sum_i (n_i -1) \le n$, we see that the query complexity due to the final bags is at most
\begin{equation}\label{eq:bagfinal}
\sum_i \frac{n_i (n_i -1)} {2} 
\le \max_i n_i \, \sum_i \frac{n_i -1} {2}
\le \frac{\kappa n}{2}
< \Delta n.
\end{equation}
The overall expected query complexity does not exceed
\begin{equation}\label{eq:compbound}
n\Delta + (\Delta + n  (4+ 6 \Delta) ) \times 2 \log_\Delta (n).    
\end{equation}
This finishes the proof of Theorem~\ref{th:main}. 

\hfill$\square$

\subsection{The main technical lemma}

\begin{lem}\label{lem:main}
For a bag $B$, and random nodes $u_1,\ldots,u_\kappa$ 
taken uniformly at random without replacement
{ from $B\setminus \{\rho(B\}$},
let $M$ be the size of the largest bag output by 
\textsc{bigsplit} $(B)$,
where $|B| \ge \kappa+1$, and $\kappa = \Delta$.
Then
$$
\P \left\{ M \ge 1+ \frac{|B|}{\sqrt{\Delta}} \right\} \le \frac{1}{2}.
$$
\end{lem}

\begin{figure}[hbt!]
\centering
\includegraphics[scale=.9]{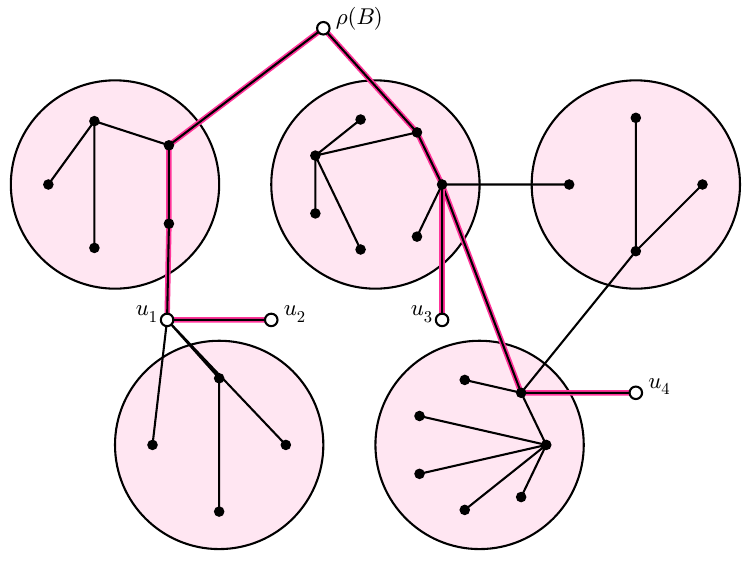}
\caption{Illustration of \textsc{bigsplit}$(B)$. The skeleton induced by $\{ \rho (B),u_1,u_2,u_3,u_4 \}$ is shown in thick lines. The remainder of the tree induced by $B$ is shown in thin lines. The nodes traversed by \textsc{dfs} between visits of $\{ \rho (B),u_1,u_2,u_3,u_4 \}$ are grouped.
}
\label{fig:proof}
\end{figure}

\begin{proof}
Consider $\rho(B)$ as the root of the tree induced by $B$, on which we perform \textsc{dfs} (depth-first-search). Note that each of the bags
left after \textsc{bigsplit} $(B, u_1,\ldots,u_\kappa)$ corresponds to a subtree with root on the skeleton output by \textsc{bigsplit} and
having root degree one.
We list the regular nodes in \textsc{dfs} order and separate this list into $\kappa+1$ sublists, all separated by $\rho(B), u_1, \ldots, u_\kappa$.
Call the sizes of the sublists $N_0, N_1, \ldots, N_\kappa$.
Note that the bags consist of regular nodes, except possibly
their representatives on the skeleton.  Each bag is either contained in a sublist or a sublist plus one of the nodes $\rho(B), u_1, \ldots, u_\kappa$.
Thus, 
$$
M \le \max_i N_i + 1.
$$
If we number the nodes by \textsc{dfs} order, starting at $\rho(B)$, then picking $k$ nodes uniformly at random without replacement from all nodes, $\rho(B)$ excepted, shows that
the $N_i$'s correspond to the cardinalities of the intervals defined by the selected nodes. Therefore,
$N_0,\ldots,N_\kappa$ are identically distributed.  In addition, for an arbitrary integer $k$,

\begin{align*}
  \P \left\{ N_0 \ge k \right\}
&= \frac{|B|-1-k}{|B|-1}  \times
\frac{|B|-2-k}{|B|-2} \times \cdots \times
\frac{|B|-\kappa-k}{|B|-\kappa} \\
&\le \left( 1 - \frac{k}{|B|-1} \right)^\kappa \\
&\le  \exp \left( - \frac{k \kappa} {|B|-1} \right).
\end{align*}
Thus, by the union bound,
\begin{align*}
  \P \left\{ M \ge 1+ \frac{|B|}{\sqrt{\Delta} } \right\} 
&\le  (\kappa + 1) \exp \left( - \frac{|B| \kappa} {(|B|-1) \sqrt{\Delta}} \right) \\
&\le  (\kappa + 1) \exp \left( - \frac{\kappa+1 } { \sqrt{\Delta}} \right) \\
&\qquad \textrm{(as $|B| \ge \kappa + 1$)} \\
&\le  (\Delta + 1) \exp \left( - \frac{\Delta+1 } { \sqrt{\Delta}} \right) \\
&\le 4 \exp \left( - \frac{4 } { \sqrt{3}} \right) \\
&\qquad \textrm{(as the expression decreases for $\Delta \ge 3$)} \\
&= 0.39728\ldots \\
&< \frac{1}{2}.
\end{align*}

\end{proof}

\subsection{More refined probabilistic bounds}

{ In this section we refine Theorem~\ref{th:main} by offering more detailed distributional estimates for the query complexity of Algorithm~\ref{alg:tree}.}

\begin{thm}\label{th:minor}
The query complexity $Z$ of Algorithm~\ref{alg:tree} with parameter $\kappa = \Delta$ has 
$$
\E \{ Z \} \;\le\; 19 \Delta n \log_\Delta (n).
$$
Furthermore, if $\log_\Delta (n) \to \infty$ as $n \to \infty$,
then $\P \{ Z \ge 2 \E \{Z\} \} = o(1)$.
Finally, if $\log_\Delta (n) = O(1)$, then
$Z/(\Delta n) 
{\stackrel{st.}{\le}}
6X + o_p(1)$, 
where ${\stackrel{st.}{\le}}$ denotes stochastic
domination, $X$ is a geometric $(1/2)$ random variable,
and $o_p(1)$ is a random variable tending to $0$ in probability.
\end{thm}

\begin{proof}
With the choice $\kappa = \Delta$, using \eqref{eq:bagsplit} for the complexity due to bag splitting and \eqref{eq:bagfinal} for the complexity due to the final treatment of bags, the algorithm's query complexity can be bounded by the random variable
$$
Z = \Delta n +
\sum_k
\sum_{\textrm{bags $B$ split in round $k$}} \left( \Delta +  X_B (2+3\Delta) ( |B|-1)  \right),
$$
where all $X_B$ are i.i.d.\ geometric $(1/2)$ random variables.
Using \eqref{eq:compbound}, we have
$$
 \E \{ Z \} \le \Delta n + (\Delta + (4 + 6 \Delta) n ) \times 2 \log_\Delta (n)
 \le 19 \Delta n \log_\Delta (n).
$$
As argued in the proof of Theorem~\ref{th:main} in each round the bags get reduced by a factor of $1/\sqrt{\Delta}$. Thus, given all bags $B$ in all levels, we see that the variance $\V \{ Z \}$
is not more than
\begin{align*}
\sum_k
\sum_{\textrm{bags $B$ split in round $k$}} &\V \{ X_B \} (2+3\Delta)^2  (|B|-1)^2  \\
&\le 
(2+3\Delta)^2 \, \sum_k  \frac{n}{\Delta^{k/2}} \sum_{\textrm{bags $B$ split in round $k$}} 2 (|B|-1)\\
&\le 
2n (2+3\Delta)^2 \, \sum_k \frac{n}{\Delta^{k/2}} \\
&\le 
\frac{ 32 (n\Delta)^2 } {1 - 1/\sqrt{\Delta}}\\
&\le 
90 (n\Delta)^2  .
\end{align*}
There are two cases: if $\log_\Delta (n) \to \infty$, then by Chebyshev's inequality, 
$$
\P \{ Z > 2 \E \{Z\} \} \to 0~,
$$
where we used the fact that  $\E(Z)=\Omega(\Delta n \log_\Delta(n))$ (\cite{zhang2003complexity}).
If on the other hand, $\log_\Delta (n) \le K$ for a large constant $K$,
i.e., $\Delta \ge n^{1/K}$, then 
$$
Z \;\le\; \Delta n + \Delta + X (2+3\Delta) n + Z'
\;\le\; 6n \Delta X + Z',
$$
where
$Z'$ is defined as $Z$ with level $k=0$ excluded.
Arguing as above, we have
$$
\V \{ Z' \}
\le \frac{ 32 (n\Delta)^2 } {\sqrt{\Delta} - 1}
= o \left( (\E \{ Z \} )^2 \right),
$$
so that by Chebyshev's inequality,
$$
\frac {Z}{n \Delta} \le 6X +o_p(1),
$$
where $o_p(1)$ denotes a quantity that tends to zero in probability
as $n \to \infty$.  In other words, the sequence of
random variables $Z/(n \Delta)$ is tight.
\end{proof}

\begin{rem}
For fixed $\Delta \ge 3$, the complexity grows as $ n\log n$.  In many cases, this is exponential in the diameter of the tree.  For example, in a complete $\Delta$-ary tree, the number of leaves is at least $n/2$, while the diameter is about $2 \log_\Delta n$. Reconstruction of such trees is impossible without performing at least $n/4$ distance computations.
\end{rem}

\section{Graphical models on semi-labeled trees}\label{sec:models}

We present a family of probabilistic models over partially observed trees for which the distance-based Algorithm~\ref{alg:tree} recovers the underlying tree. Although the Gaussian and the binary tree models (also known as the Ising tree model) discussed in Section~\ref{sec:binGaus} are probably the most interesting, with little effort we can generalize this, which we do in Section~\ref{sec:more} and Section~\ref{sec:more2}.

Given a tree $T=(V, E)$ and a random vector $Y=(Y_v)_{v\in V}$ with values in the product space $\cY=\prod_{v\in V}\cY_v$, consider the underlying graphical model over $T$, that is, the family of density functions that factorize according to the tree
\begin{equation}\label{eq:fY}
f_Y(y)\;=\;\prod_{uv\in E} \psi_{uv}(y_u,y_v)\qquad\mbox{for }y\in \cY,	
\end{equation}
where $\psi_{uv}$ are non-negative potential functions (\cite{lau96}). The underlying \emph{latent tree model} over the semi-labeled tree $T$ with the labelling set $W=[n]$ is a model for $X=(X_1,\ldots, X_n)$, which is the sub-vector of $Y=(X, H)$ associated to the regular vertices. The density of $X$ is obtained from the joint density of $Y$ by marginalizing out the latent variables $H$
$$
f(x)\;=\;\int_{\cH} f_Y(x,h){\rm d} h.
$$

Note that in the definition of a semi-labeled tree, we required that all nodes of $T$ of degree $\leq 2$ are regular. The restriction complies with this definition; c.f. Section~11.1 in \cite{zwiernik2018latent}.

\subsection{The binary/Gaussian case}\label{sec:binGaus}

We single out two simple examples where $Y$ is jointly Gaussian or when $Y$ is binary. In the second case, the model on a tree coincides with the binary Ising model on $T$. In both cases, we get a useful path product formula for the correlations, which states that the correlation between any two regular nodes can be written as the product of edge correlations for edges on the path joining these nodes
\begin{equation}\label{eq:pathpr}
\rho_{ij}\;=\;\prod_{uv\in \overline{ij}} \rho_{uv}\qquad \mbox{for all }i,j\in [n].	
\end{equation}
The advantage of this representation is that \eqref{eq:pathpr} gives a direct translation of correlation structures in these models to tree metrics via $d({i,j}):=-\log |\rho_{ij}|$ for all $i\neq j$; see also \cite[Section 8.3.3]{pearl}. To make this explicit we formulate the following proposition.
\begin{prop}
Consider a latent tree model over a semi-labeled tree $T$. Whenever the correlations $\Sigma=[\rho_{ij}]$ in the underlying tree model admit the path-product formula~\eqref{eq:pathpr}, Algorithm~\ref{alg:tree}, applied to the distance oracle defined by $d({i,j})=-\log |\rho_{ij}|$ for all $i\neq j$, recovers the underlying tree.  
\end{prop}

Note that our algorithm recovers all the edge lengths, so we can also find the absolute values of the correlations between the latent variables. In the Gaussian case, this yields parameter identification.
\begin{rem}
In the Gaussian latent tree model over a semi-labeled tree $T$, Algorithm~\ref{alg:tree} recovers the underlying tree and the model parameters up to sign swapping of the latent variables. 
\end{rem}

It turns out that the basic binary/Gaussian setting can be largely generalized. We discuss three such generalizations:
\begin{enumerate}
	\item [(1)] General Markov models
	\item [(2)] Linear models
	\item [(3)] Non-paranormal distributions
\end{enumerate}

\noindent We briefly describe these models for completeness. The former two are dealt with in Section~11.2 in \cite{zwiernik2018latent}. 

\subsection{General Markov models and linear models}\label{sec:more}

By general Markov model, we mean a generalization of the binary latent tree models, where each $\cY_v=\{0,\ldots,d-1\}$ for some finite $r\geq 1$. Denoting by $P^{uv}$ the $d\times d$ matrix representing the joint distribution of $(Y_u, Y_v)$, and by $P^{vv}$ the diagonal $d\times d$ matrix with the marginal distribution of $Y_v$ on the diagonal, we can define for any two nodes $u,v$
\begin{equation}\label{eq:tauGMM}
\tau_{uv}\;:=\;\frac{\det(P^{uv})}{\sqrt{\det(P^{uu})\det(P^{vv})}}.
\end{equation}

It turns out that for these new quantities, an equation of type \eqref{eq:pathpr} still holds, namely, $\tau_{ij}=\prod_{uv\in \overline{ij}}\tau_{uv}$, so we again obtain the tree distance $d({i,j})=-\log|\tau_{ij}|$. 

\begin{prop}
	In a general Markov model over a semi-labeled tree $T$ we can recover $T$ using Algorithm~\ref{alg:tree} from the distances $d({i,j})=-\log|\tau_{ij}|$ with $\tau_{ij}$ defined in \eqref{eq:tauGMM}.
	\end{prop}

More generally, suppose that $Y_v$ are now potentially vector-valued, all in $\R^k$, and it holds for every edge $uv$ of $T$ that $\E[Y_u|Y_v]$ is an affine function of $Y_v$. Let 
$$
\Sigma_{uv}\;=\;{\rm cov}(Y_u,Y_v)\;=\;\E Y_u Y_v^\top -\E Y_u \E Y_v^\top.
$$In this case, defining $\tau_{uv}$ as 
\begin{equation}\label{eq:tauLM}
\tau_{uv}\;:=\;\det(\Sigma_{uu}^{-1/2}\Sigma_{uv}\Sigma^{-1/2}_{vv}),
\end{equation}
 we reach the same conclusion as for general Markov models; see Section~11.2.3 in \cite{zwiernik2018latent} for details.

\begin{prop}
	In a linear model over a semi-labeled tree $T$ we can recover recover $T$ using Algorithm~\ref{alg:tree} from the distances $d({i,j})=-\log|\tau_{ij}|$ with $\tau_{ij}$ defined in \eqref{eq:tauLM}.
	\end{prop}

\subsection{Non-paranormal distributions}\label{sec:more2}

Suppose that $Z=(Z_v)$ is a Gaussian vector 
{
with underlying latent tree $T=(V, E)$.} Suppose that $Y$ is a monotone transformation of $Z$, so that $Y_v=f_v(Z_v)$ for strictly increasing functions $f_v$, $v\in V$. The conditional independence structure of $Y$ and $Z$ are the same. The problem is that the correlation structure of $Y$ may be quite complicated and may not satisfy the product path formula in \eqref{eq:pathpr}. Suppose however that we have access to the Kendall-$\tau$ coefficients $K=[\kappa_{ij}]$ for $X$ with 
\begin{equation}\label{eq:kt}
\kappa_{ij}\;=\;{\rm corr}({\rm sgn}(X_i-X_i'),{\rm sgn}(X_j-X_j')),
\end{equation}
where $X'$ is an independent copy of $X$. Then we can use the fact that for all $i\neq j$
$$
{\rm corr}({\rm sgn}(X_i-X_i'),{\rm sgn}(X_j-X_j'))\;=\;{\rm corr}({\rm sgn}(Z_i-Z_i'),{\rm sgn}(Z_j-Z_j'))\;=:\;\kappa^Z_{ij}.
$$
As observed by \cite{liu2012high}, since $Z$ is Gaussian, we have a simple formula that relates $\kappa^Z_{ij}$ to the correlation coefficient $\rho_{ij}={\rm corr}(Z_i,Z_j)$, for all $i\neq j$
$$
\rho_{ij}\;=\;\sin(\tfrac\pi2 \kappa^Z_{ij})\;=\;\sin(\tfrac\pi2 \kappa_{ij}).
$$ 
Applying a simple transformation to the oracle $K$ gives us access to the underlying Gaussian correlation pattern, which now can be used as in the Gaussian case. 
\begin{prop}
	In a non-paranormal distribution over a semi-labeled tree $T$ we can recover $T$ using Algorithm~\ref{alg:tree} from the distances $$d({i,j})\;:=\;-\log|\sin(\tfrac\pi2\kappa_{ij})|,$$ where $\kappa_{ij}$ is the Kendall-$\tau$ coefficient defined in \eqref{eq:kt}.
	\end{prop}

\section{Statistical guarantees}

In this section, we illustrate how the results developed in this paper can be applied 
{ in a more realistic scenario when the entries of the covariance matrix cannot be measured exactly. In particular, }suppose a random sample of size $N$ is observed from a zero-mean distribution with covariance matrix $\Sigma$. Assume that $\Sigma_{ii}=1$ for all $i=1,\ldots, n$ and the correlations $\Sigma_{ij}=\rho_{ij}$ satisfy parametrization \eqref{eq:pathpr} for some semilabeled tree. In this case $d(i,j)=-\log|\rho_{ij}|$ for all $i\neq j$ forms a tree metric. 

Denote by $\widehat\rho_{ij}$ a suitable estimator of the correlations based on the sample and let $\widehat d(i,j)=-\log|\widehat\rho_{ij}|$ be the corresponding plug-in estimator of the distances. Since the  $\widehat d(i,j)$ do not form a tree metric, we cannot apply directly Algorithm~\ref{alg:tree}. The algorithm uses distances at { five} places:
\begin{enumerate}
    \item In Algorithm~\ref{alg:basic}, to compute $D(v)$.
    \item In Algorithm~\ref{alg:basic}, to decide if $\rho(B_i)$ is latent or not.
    \item In the case when $\rho(B_i)$ is latent, Algorithm~\ref{alg:basic} also uses the distances to calculate $d(\rho(B_i),v)$ for all $v\in B_i$.
    \item In Algorithm~\ref{alg:explode}, to group nodes in $B$ according to the connected components obtained by removing $\rho(B)$.
    \item When recovering the skeleton of a small subtree in the case when $|B|\leq \kappa$.
\end{enumerate}

{ In order to adapt the algorithm to the ``noisy'' distance oracle,}
we first propose the noisy versions of \textsc{basic} and \textsc{explode}. The procedure \textsc{basic}.noisy is outlined in Algorithm~\ref{alg:basicnoisy}, and \textsc{explode}.noisy in Algorithm~\ref{alg:explodenoisy}. The performance of the procedure depends on the following quantities
\begin{equation}\label{eq:LU}
\ell\;:=\;\min_{i\neq j} d({i,j}),\quad\qquad \mathit{u}\;:=\;\max_{i\neq j} d({i,j}),	
\end{equation}
where the minimum and maximum are taken over all regular nodes.

{ The algorithms have an additional input parameter $\epsilon>0$ that is an upper bound for the noise level. More precisely, the algorithms work correctly whenever $\max_{i\neq j} |\widehat d(i,j)-d(i,j)| \le \epsilon$.
}

\begin{algorithm}[H]
Input: a bag $B$, $\alpha\in B$, and $\epsilon>0$\;
\For{$v\in B$}{
Compute $\widehat D(v)=\widehat d({ v,\alpha})-\widehat d({v,\rho})$\; 
}
Order $v\in B$ according to the decreasing value of $\widehat D(v)$\;
If $|\widehat D(u)-\widehat D(v)|\leq  4\epsilon$, assign $u,v$ to the same bag\;
Denote the resulting bags by $B_1,\ldots, B_{k}$\;
\For{$i=1,\ldots,k$}{
$u_i=\arg\min_{v\in B_i} \widehat d(\alpha,v)$\;
\If{$|\widehat d(\rho,u_i)+\widehat d(u_i,\alpha)-\widehat d(\rho,\alpha)|\leq 3\epsilon$}{$\rho(B_i)=u_i$}
\Else{Identify latent node $w_i$\;
Add $w_i$ to $B_i$\;
Set $\rho(B_i)=w_i$\;
{ 
Calculate } $\widehat d(w_i,v)$ for all $v\in B_i$ using \eqref{eq:newreg}\;}
}
Return $B_1,\ldots,B_k$ and the skeleton $(\rho(B_1),\rho(B_2)),\ldots,(\rho(B_{k-1},\rho(B_k)))$\;
%\vspace*{.2cm}
\label{alg:basicnoisy}
\caption{\textsc{basic}.noisy$(B,\alpha,\epsilon)$}
\end{algorithm}

The next simple fact is used repeatedly below.
\begin{lem}\label{lem:ud}
	Suppose $\max_{i\neq j}|\widehat d(i,j)-d(i,j)|\leq \epsilon$ for all regular $i\neq j$ and let $a\in \R^{n\choose 2}$. Then $|a^\top (\widehat d-d)|\leq \|a\|_1\epsilon$. In particular,
	\begin{enumerate}
		\item [(i)] if $a^\top d=0$ then 
	$|a^\top \widehat d|\leq\|a\|_1 \epsilon$;
	\item [(ii)] if $a^\top d\geq \eta$ then $a^\top \widehat d\geq \eta-\|a\|_1\epsilon$.
	\end{enumerate}
\end{lem}

\begin{algorithm}[H]\label{alg:explodenoisy}
Input: a bag $B$ \;
\While{$|B|>1$}
{Take $u \in B$, $u \not= \rho(B)$\;
Set $B^* = \{ v \in B, v\not=\rho(B):   d(u,\rho(B))+d(v,\rho(B))-d(u,v)>3\epsilon \} \cup \{ \rho(B) \}$\;
Set $\rho(B^*) = \rho(B)$\;
Output $B^*$ \;
Set $B \leftarrow \{ \rho (B) \} \cup ( B \setminus B^* )$\;
}
\caption{\textsc{explode}.noisy$(B,\epsilon)$}
\end{algorithm}

In what follows, we condition on the random event
\begin{eqnarray}
\cE(\epsilon)&=&\left\{\max_{i\neq j}|\widehat d(i,j)-d(i,j)|\leq \epsilon\;\;\mbox{for all regular }i\neq j\right\}\\
\nonumber&=&\left\{\max_{i\neq j}\left|\log \left|\frac{\widehat\rho_{ij}}{\rho_{ij}}\right|\right|\leq \epsilon\mbox{ for all regular }i\neq j\right\}.
\end{eqnarray}
\begin{prop}\label{prop:noisyasnoiseless}
Suppose that in the current round, the event $\mathcal E(\epsilon)$ holds with $\epsilon<\ell/4$. Then Algorithm~\ref{alg:basicnoisy} and Algorithm~\ref{alg:explodenoisy} applied to the noisy distances gives the same output as Algorithm~\ref{alg:basic} and Algorithm~\ref{alg:explode} applied to their noiseless versions.	
\end{prop}
\begin{proof}
    In the first part, Algorithm~\ref{alg:basicnoisy} computes $\widehat D(v)$ for all $v$ and it uses this information to produce bags $B_1,\ldots,B_{k}$. The bags in Algorithm~\ref{alg:basic} are obtained by grouping nodes based on the increasing values of $D(v)$. By Lemma~\ref{lem:ud}(i), if $D(u)=D(v)$ then 
	$|\widehat D(u)-\widehat D(v)|\;\leq\;4\epsilon$. Moreover, if $D(u)>D(v)$ then it must be that $D(u)-D(v)\geq 2\ell$ and so, by Lemma~\ref{lem:ud}(ii), $\widehat D(u)-\widehat D(v)>2\ell-4\epsilon>4\epsilon$. This shows that this step of Algorithm~\ref{alg:basicnoisy} provides the same bags as Algorithm~\ref{alg:basic}. In the second part of the algorithm, we decide whether or not the corresponding path nodes $\rho(B_i)$ are regular. Let $\hat u_i=\arg\min_{v\in B_i} \widehat d(\alpha,v)$ and $u_i=\arg\min_{v\in B_i} d(\alpha,v)$. If $\hat u_i\neq u_i$ then $d({\rho,\hat u_i})-d({\rho,u_i})\geq \ell$. By Lemma~\ref{lem:ud}(ii), 
	$$
	\widehat d({\rho,\hat u_i})-\widehat d({\rho,u_i})\;\geq\;\ell-2\epsilon\;\geq \;2\epsilon,
	$$
	which contradicts the fact that $\hat u_i=\arg\min_{v\in B_i} \widehat d({\alpha,v})$. We conclude that $\hat u_i=u_i$.	Now consider the problem of deciding whether $\rho(B_i)=u_i$. Suppose first that $d({\rho, u_i})+d({\alpha, u_i})-d({\rho,\alpha})=0$ (i.e., $\rho(B_i)=u_i$). By Lemma~\ref{lem:ud}(i),
	$$
	|\widehat d({\rho, u_i})+\widehat d({\alpha, u_i})-\widehat d({\rho,\alpha})|\;\leq\;3\epsilon.
	$$
	If $d({\rho, u_i})+d({\alpha, u_i})-d({\rho,\alpha})>0$ then, since $D$ is a tree metric,  it must be $d({\rho, u_i})+d({\alpha, u_i})-d({\rho,\alpha})\geq 2 \ell$.  Then, by Lemma~\ref{lem:ud}(ii) and by the fact that $\epsilon<\ell/4$
	$$
	\widehat d({\rho, u_i})+\widehat d({\alpha, u_i})-\widehat d({\rho,\alpha})\;\geq\;2\ell-3\epsilon\;> \;5\epsilon.
	$$
 This shows the correctness of the second part of Algorithm~\ref{alg:basicnoisy}. 
 
 Since $d(v,\alpha)+d(v,\rho)-d(\alpha)\geq 2\ell$, we get $\widehat d(v,\alpha)+\widehat d(v,\rho)-\widehat d(\alpha)\geq 2\ell-3\epsilon$
 
We can similarly show that Algorithm~\ref{alg:explodenoisy} gives the same output as Algorithm~\ref{alg:explode}.
\end{proof}

The problem with applying Proposition~\ref{prop:noisyasnoiseless} recursively to each round is that the event $\cE(\epsilon)$ only bounds the noise for distances between the $n$ originally regular nodes. As the procedure progresses, new nodes are made regular; if $\rho(B_i)$ is latent we make it ``regular'' by updating distances $\widehat d(w_i,v)$ for all $v\in B_i$ using \eqref{eq:newreg}. 
\begin{lem}\label{lem:bounddist}
    Suppose that event $\cE(\epsilon)$ holds. After $R$ rounds of the algorithm,  we have $|\widehat d(i,j)-d(i,j)|\leq (1+\tfrac{R}{2})\epsilon$ for all $i,j\in B$ that are regular in the $R$-th round. 
\end{lem}
\begin{proof}
Consider the first run of Algorithm~\ref{alg:basicnoisy}. In this case, $\rho(B)$ and $\alpha$ are both regular. If $w_i$ is identified as a latent node, we calculate $\widehat d(w_i,v)$ for all $v\in B_i$ using \eqref{eq:newreg}. By the triangle inequality,  $|\widehat d(w_i,v)- d(w_i,v)|\leq \tfrac32\epsilon$ and this bound is sharp. This establishes the case $R=1$. Suppose the lemma bound holds up to the $(R-1)$-st round. In the $R$-th round, $\rho(B)$ may be a latent node added in the previous call but $\alpha$ is still sampled from the originally regular nodes. If $\rho(B_i)$ is identified as a latent node, we calculate $\widehat d(\rho(B_i),v)$ for all $v\in B_i$. Since $v$ is also among the originally regular nodes, we conclude $|\widehat d(v,\alpha)-d(v,\alpha)|\leq \epsilon$. By induction, $|\widehat d(v,\rho(B))-d(v,\rho(B))|\leq (1+\tfrac{R-1}{2})\epsilon$ and $|\widehat d(\alpha,\rho(B))-d(\alpha,\rho(B))|\leq (1+\tfrac{R-1}{2})\epsilon$. By the triangle inequality,
$$
|\widehat d(\rho(B_i),v)-d(\rho(B_i),v)|\;\leq\;\tfrac12(\epsilon+(1+\tfrac{R-1}{2})\epsilon+(1+\tfrac{R-1}{2})\epsilon)\;=\;(1+\tfrac{R}{2})\epsilon.
$$
The result now follows by induction.
\end{proof}

It is generally easy to show that the event $\mathcal E(\epsilon)$ holds with probability at least $1-\eta$ as long as the sample size $N$ is large enough. Let $\delta=1-e^{-\epsilon}$ and suppose that the following event holds
$$
\mathcal E'(\delta)\;:=\;\left\{ |\widehat\rho_{ij}-\rho_{ij}|\leq |\rho_{ij}|\delta\;\;\mbox{ for all regular }i\neq j\right\}.
$$
Since $\delta<1$, under the event $\mathcal E'$ the signs of $\widehat\rho_{ij}$ and $\rho_{ij}$ are the same. It is easy to see that $\mathcal E'\subseteq \mathcal E$. Indeed, under $\mathcal E'$, $\frac{\hat \rho_{ij}}{\rho_{ij}}>0$ and for all $1-\delta\leq x\leq 1+\delta$ we have $\log(1-\delta)\leq \log(x)\leq \log(1+\delta)$. It follows that 
$$
\left|\log\frac{\hat \rho_{ij}}{\rho_{ij}}\right|\;\leq\;\max\left\{\log(1+\delta),|\log(1-\delta)|\right\}\;=\;-\log(1-\delta)\;=\;\epsilon.
$$
It is then enough to bound the probability of the event $\cE'$. To illustrate how this can be done without going into unnecessary technicalities, suppose $\max_i \E X_i^4 \leq  \kappa$ for some $\kappa > 0$. In this case ${\rm var}(X_i X_j) \leq  \kappa$, and therefore one may use the median-of-means estimator 
(see, e.g., \cite{lugosi2019mean}) to estimate $\rho_{ij} = \E[X_i X_j]$. We get the following result.

\begin{thm}\label{th:scomp}Suppose a random sample of size $N$ is generated from a mean zero distribution with covariance matrix $\Sigma$ satisfying $\Sigma_{ii}=1$ and suppose $\max_i \E X_i^4 \leq  \kappa$ for some $\kappa > 0$. Let $\ell,\mathit{u}$ be as defined in \eqref{eq:LU}.	Fix $\eta\in (0,1)$ and suppose 
$$\epsilon\;\leq\; \frac{\ell}{4(1+\log_\Delta(n))}\qquad\mbox{and}\qquad N\;\geq \; \frac{64\kappa \log(n/\eta)}{ (1-2e^{-\epsilon})}e^{2\mathit u}$$  
then the noisy version of the Algorithm~\ref{alg:tree} correctly recovers the underlying semi-labeled tree with probability $1-\eta$. 
\end{thm}
\begin{proof}
Suppose the event $\cE(\epsilon)$ holds. Like in the proof of Theorem~\ref{th:main} we modify the procedure by repeating  \textsc{bigsplit}.noisy until the largest bag is bounded in size by $|B|/\sqrt{\Delta}$. With this modification, the algorithm stops after each round. With our choice of $\epsilon$, by Lemma~\ref{lem:bounddist}, we are guaranteed that all computed distances satisfy $|\widehat d(u,v)-d(u,v)|\leq \epsilon$ in the first $2\log_\Delta(n)$ rounds. By Proposition~\ref{prop:noisyasnoiseless}, all these subsequent calls of \textsc{bigsplit}.noisy return the same answer as \textsc{bigsplit} applied to noiseless distances.  The proof now follows if we can show that, with probability at least $1-\eta$, event $\cE(\epsilon)$ holds. We  show that $\cE'(\delta)$ with $\delta=1-e^{-\epsilon}$ holds, which is a stronger condition. Recall that $\E X_i^2=1$, $\E X_i^4\leq \kappa$ and, in consequence, $\var(X_i X_i)\leq \kappa$. We want to estimate $\E(X_i X_j)=\rho_{ij}$. By Theorem 2 in \cite{lugosi2019mean}, the median-of-means estimator $\widehat\rho_{ij}$ (with an appropriately chosen number of blocks that depends on $\eta$ only) satisfies that with probability at least $1-2\eta/{n \choose 2}$
	$$
	|\widehat\rho_{ij}-\rho_{ij}|\;\leq\;\sqrt{\frac{32\kappa \log({n\choose 2}/{\eta})}{N}}.
	$$ 
	Thus, we get that with probability at least $1-\eta$, all $\widehat\rho_{ij}$ satisfy simultaneously that $|\widehat\rho_{ij}-\rho_{ij}|\leq |\rho_{ij}|\delta$ as long as
	$$
	N\;\geq\;\frac{64\kappa \log(n/\eta)}{\delta^2 \min_{i\neq j}\rho_{ij}^2 }\;=\;\frac{64\kappa \log(n/\eta)}{e^{-2\mathit u} \delta^2}.
	$$
	The inequality $\epsilon<\tfrac{\ell}{4(1+\log_\Delta(n))}$ is equivalent to $\delta<1-e^{-\tfrac{\ell}{4(1+\log_\Delta(n))}}$. Thus, we require
	$$
	N\;\geq\;\frac{64\kappa \log(n/\eta)}{e^{-2\mathit u} (1-2e^{-\tfrac{\ell}{4(1+\log_\Delta(n))}})}.
	$$
\end{proof}

Note that the required sample size $N$ grows exponentially with the diameter $u$ of the tree. While this seems undesirable, in the setup of this section, such a dependence is easily seen to be necessary, as the correlations along any path decrease exponentially, and any estimator of the correlations needs to have accuracy at the same scale.

%%%%%%%%%%%%%%%%%%%%%%%%%%%%%%%%%%%%%%%%%%%%%%
%% Funding information, if any,             %%
%% should be provided in the                %%
%% funding section.                         %%
%%%%%%%%%%%%%%%%%%%%%%%%%%%%%%%%%%%%%%%%%%%%%%
\begin{funding}
The authors thank the referees for their insightful comments.
Luc Devroye was supported by NSERC grant RGPIN-2024-04164. Piotr Zwiernik and G\'abor Lugosi acknowledge the support of Ayudas Fundaci\'{o}n BBVA a
Proyectos de Investigaci\'{o}n Cient\'{i}fica 2021 and the Spanish Ministry of
Economy and Competitiveness grant PID2022-138268NB-I00, financed by
MCIN/AEI/10.13039/501100011033, FSE+MTM2015-67304-P, and FEDER, EU. Piotr Zwiernik was also supported by NSERC grant RGPIN-2023-03481.\end{funding}

%%%%%%%%%%%%%%%%%%%%%%%%%%%%%%%%%%%%%%%%%%%%%%
%% Supplementary Material, including data   %%
%% sets and code, should be provided in     %%
%% {supplement} environment with title      %%
%% and short description. It cannot be      %%
%% available exclusively as external link.  %%
%% All Supplementary Material must be       %%
%% available to the reader on Project       %%
%% Euclid with the published article.       %%
%%%%%%%%%%%%%%%%%%%%%%%%%%%%%%%%%%%%%%%%%%%%%%
%\begin{supplement}
%\stitle{Title of Supplement A}
%\sdescription{Short description of Supplement A.}
%\end{supplement}
%\begin{supplement}
%\stitle{Title of Supplement B}
%\sdescription{Short description of Supplement B.}
%\end{supplement}

%% or include bibliography directly:
\bibliographystyle{imsart-nameyear}\bibliography{references2.bib}

\end{document}